\documentclass{amsart} 

\usepackage{amssymb,latexsym,amsmath}
\def\C{\mathbb C}
\def\R{\mathbb R}
\def\N{\mathbb N}
\def\Z{\mathbb Z}

\newtheorem{thm}{Theorem}[section]
\newtheorem{lem}{Lemma}[section]

\newtheorem{definition}{Definition}[section]

\newtheorem{prop}{Proposition}[section]

\begin{document}
\sffamily
\title{Trajectories escaping to infinity in finite time}
\author{J.K. Langley}
\address{School of Mathematical Sciences 
\\University of Nottingham 
\\NG7 2RD, UK}
\email{james.langley@nottingham.ac.uk}

\begin{abstract}
If the function $f$ is transcendental and  meromorphic in the plane, and either $f$ has finitely many poles or
its inverse function has a logarithmic singularity over $\infty$,
then the equation $\dot z = f(z)$ has infinitely many trajectories  tending to infinity in finite increasing time. 
\end{abstract}
\maketitle

\section{Introduction}

This paper concerns the differential equation 
\begin{equation}
 \label{flow1}
\dot z = \frac{dz}{dt} = f(z),
\end{equation}
in which the function $f$ is meromorphic in a plane domain $D$: see, for example,
\cite{brickman,2,garijo,hajek1,hajek2,hajek3, kingneedham} for fundamental results concerning such flows.  
A trajectory for (\ref{flow1})
is a path $z(t)$ in $D$ with  $z'(t) = f(z(t)) \in \C$ for $t$ in some maximal interval $(\alpha, \beta)  \subseteq \R$.
The present paper is motivated by a result  from \cite{kingneedham}
involving
trajectories 
which tend to infinity in finite increasing 
time, that is, which satisfy $\beta \in \R$ and $\lim_{t \to \beta - } z(t) = \infty$.
King and Needham \cite[Theorem 5]{kingneedham} showed that if  $f$ has a pole at 
infinity of order at least $2$  then such trajectories always exist for (\ref{flow1})
(see Section~\ref{prelim}).
It seems reasonable to ask whether trajectories of this type must exist if  $f$ is  transcendental and meromorphic in the plane, 
and the following will be proved in Section~\ref{nonzero}.

\begin{thm}
 \label{thm0}
Let the function $f$ be  transcendental and meromorphic in the plane, with finitely many poles. 
Then (\ref{flow1}) has infinitely many pairwise disjoint trajectories each  tending to
infinity in finite increasing time. 
\end{thm}
The proof of Theorem \ref{thm0} is based on the Wiman-Valiron theory \cite{Hay5}, which shows that if $f$ is as in the hypotheses 
then there exist small neighbourhoods on which $f(z)$ behaves like a constant multiple of a large power of $z$.
In the simple example $\dot z = - \exp( -z)$, all trajectories satisfy
$\exp(z(t)) = \exp(z(0)) -t$ and tend to infinity as $t$ increases, taking finite time to do so
if and only if $\exp( z(0))$ is real and positive.  

For meromorphic functions with infinitely many poles
the situation is in general different. Let $g$ be a transcendental entire function of order of growth $\rho(g) < 1/2$ and 
let $f = -ig/g'$ in (\ref{flow1}). 
Then each trajectory has $i \log g(z(t)) = t + C $, and $\log |g(z(t))| = {\rm Im} \, C$, 
with $C$ constant. Because $\rho(g)< 1/2$, the classical $\cos \pi \rho$ theorem \cite{Hay7}
implies that 
$\min \{ |g(z)| : |z| = r \} $ is unbounded as $r \to \infty$, and so
all trajectories are bounded.
In this example $\infty$ is an asymptotic value of $f$, since   estimates for logarithmic
derivatives from \cite{Gun2} imply that $g'(z)/g(z)$ tends to $0$ as $z$ tends to infinity outside a small exceptional set. 
However, 
Theorem~\ref{thm2} below will show that infinitely many disjoint
trajectories tending to infinity in finite increasing time must exist if $f$ satisfies
the stronger condition that the inverse function has a logarithmic singularity over $\infty$, which is defined as follows \cite{BE,Nev}.

Let $f$ be any transcendental meromorphic function in the plane, let $M$ be real and positive, and let $U$ be a component of the set
$\{ z \in \C : |f(z)| > M \}$ with the following property: for some $z_0 \in U$ with $w_0 = f(z_0) \in \C$,
a branch of the inverse function 
$z = f^{-1}(w)$ is defined near $w_0$,  mapping $w_0$  to $z_0$, and admits 
unrestricted analytic continuation in the annulus
$M < |w| < \infty$. If $v_0$ is chosen so that $e^{v_0} = w_0$ then a function
$\phi(v) =  f^{-1}( e^v) $ may be defined on a neighbourhood of $v_0$ and 
extends by the monodromy theorem to an analytic function on the half-plane $H$ given by ${\rm  Re } \, v > \log M$.
Then, by a well known classification theorem \cite[p.287]{Nev}, 
there are two possibilities. First, if $\phi$ is not univalent on $H$ then $\phi$ has 
period
$m 2 \pi i$ for some minimal positive integer $m$ and $U$ contains precisely one pole $z_1$ of $f$ of multiplicity $m$, while $z \to z_1$ as
$f(z) \to \infty $ with $z \in U$. 
On the other hand, if $\phi$ is univalent on $H$, then $U$ contains no poles of $f$, but $U$ does contain a path tending to infinity
on which $f(z)$ tends to infinity, and the inverse function $\log f$ of $\phi$
maps $U$ univalently onto $H$. In this second case 
the inverse function of $f$ is said to have a logarithmic 
singularity over $\infty$, 
and $U$
is called a neighbourhood of the singularity~\cite{BE}.

\begin{thm}
 \label{thm2}
Let the function $f$ be transcendental and meromorphic in the plane such that its inverse function $f^{-1}$ has a logarithmic 
singularity over $\infty$.  Then for each neighbourhood $U$ of the singularity there exist
infinitely many  pairwise disjoint trajectories of the flow (\ref{flow1}), on each of which 
$z(t)$ tends to infinity in finite increasing time with $z(t) \in U$. 
\end{thm}

Examples to which Theorem \ref{thm2} applies include
$ f(z) = e^{-z^2} \tan z $: here $\infty$ is an asymptotic value of $f$, but 
the finite critical and asymptotic values  form a bounded set, so that the two singularities of $f^{-1}$ over $\infty$ are logarithmic.
Theorem \ref{thm2} follows from the next result.

\begin{thm}\label{thm1}
 Let $f$ be a meromorphic function on a domain $\Omega \subseteq \C$ 
such that there exist a real number $M > 0$, a domain 
$U \subseteq \Omega$ and an analytic function $F: U \to \C$
with the property that $f = e^F$ on $U$ and $F$ maps $U$ univalently onto the half-plane $H = \{ w \in \C : {\rm Re}\, w > \log M \} $.
Then 
(\ref{flow1})
has infinitely many  pairwise disjoint trajectories $z(t)$ 
on which 
$z(t)$ tends  in finite increasing time  from within $U$ to the extended boundary $\partial_\infty \Omega$
of $\Omega$. 
\end{thm}

Here the statement that a trajectory $z(t)$ tends  in finite increasing time from within $U$ to the extended boundary of $\Omega$
means that
there exists $T \in \R$ with the following property:
to each compact set $K_0 \subseteq \Omega$ corresponds $t_0 \in (- \infty,  T)$ with $z(t) \in U \setminus K_0$ for $t_0 < t < T$. 
To deduce Theorem \ref{thm2} from Theorem \ref{thm1} it is only necessary to  
take $\Omega = \C$ and $U$ to be a neighbourhood of the logarithmic singularity of $f^{-1}$ over $\infty$, so that
$|z(t)| \to + \infty$ as $t \to T-$. 

\section{Preliminaries}\label{prelim}

If the function $f$ is meromorphic and non-constant on a  domain $D \subseteq \C$, and 
$w \in D$ with $f(w) \neq \infty$, then the trajectory 
of  (\ref{flow1}) through $w $ 
is the path $z(t)=\zeta_w(t) \in D$ with $z(0) = w$ and $z'(t) = f(z(t)) \in \C$ for $t$ in some maximal interval $(\alpha, \beta)  \subseteq \R$.
If $f(w) = 0$ then  $\zeta_w(t) = w$ for all $t \in \R$.
When $f(w) \neq 0$ the trajectory passes through no zeros of $f$, 
and is either simple (that is, $\zeta_w(t)$ is injective on $(\alpha, \beta) $)
or periodic (in which case $(\alpha, \beta)=\R $). 

Some standard facts concerning (\ref{flow1}) near poles  of $f$ will now be summarised:
for details, see \cite{brickman,hajek1,kingneedham}. 
If $f(z) \sim c (z-z_0)^{-m}$ as $z \to z_0$, for some $c \neq 0$ and $m \geq 0$, 
then a conformal mapping $w = \phi(z)$ is defined near $z_0$ 
by $\phi(z)^{m+1} = \int_{z_0}^z 1/f(u) \, du $, which 
gives $(m+1) w^m \dot w = 1$ and $w^{m+1}(t) = w^{m+1}(0) + t$. The equation for $w$ has  $m+1$ pairwise disjoint
trajectories tending
to $0$ in increasing time, determined by choosing $w^{m+1}(0) \in (-\infty, 0) \subseteq \R$. Thus (\ref{flow1}) has precisely $m+1$ trajectories
tending to $z_0$ in increasing time (each taking finite time to do so). 

If $D$ contains an annulus $R < |z| < \infty$ and $f$ has a pole of order $n \geq 2 $ at infinity, then setting $w = 1/z$ gives 
$\dot w = g(w) = -f(z)/z^2$, so that $g$ has a pole of order $n-2$ at $w=0$ 
and (\ref{flow1}) has $n-1$ trajectories tending to infinity 
in finite increasing time: this proves the result of King and Needham \cite{kingneedham} referred to in the introduction.

Theorem \ref{thm0} requires the following lemma:
a proof is included for completeness.

\begin{lem}
 \label{limptlemmero}
Let the function $f$ be meromorphic and non-constant on  $ \C$. 
Let $z(t)$ be a trajectory of (\ref{flow1}), with maximal interval of definition 
$(a_0, b_0) \subseteq \R$, and assume that $b_0 < \infty$. 
Then $\lim_{t \to b_0-} z(t) $ exists and is either $\infty$ or a pole of $f$. 
\end{lem}

\begin{proof} Following \cite{hajek2}, a point $z_0 \in \C \cup \{ \infty \}$ is
called a limit point of $z(t)$ as $t \to b_0-$ if there exist 
$s_n \in (a_0, b_0)$ with $s_n \to b_0-$ and
$z(s_n) \to z_0$ as $n \to \infty$. 
Suppose that $z_0 \in \C$ with 
$f(z_0) \neq 0, \infty $ is such a limit point. Writing $u(t) = \phi(z(t)) $,
where $\phi(z) = \int_{z_0}^{z} 1/f(s) \, ds$, transforms
(\ref{flow1}) near $z_0$ to $\dot u = 1$. Let $\rho$ be small and positive and let 
$U = \phi^{-1}(B(0, 2 \rho )) $ and $V = \phi^{-1}(B(0,  \rho ))$, with $B(a, r)$ the open disc
of centre $a$ and radius $r$. Then any trajectory of (\ref{flow1}) which meets $V$ must subsequently 
travel from the boundary of $V$ to that of $U$, taking time at least $\rho$ to do so. 
Since $b_0$ is finite this implies that $z(t) \to z_0$ as $t \to b_0-$, and that the trajectory extends 
beyond time $t = b_0$, contrary to assumption. 
Thus any finite limit point $z_0$ of $z(t)$ as $t \to b_0-$ has $f(z_0) \in \{ 0, \infty \}$. 

It follows that if $z_0 \in \C \cup \{ \infty \}$ is a limit point of $z(t)$ as $t \to b_0-$, 
then $\lim_{t \to b_0-} z(t) = z_0$. If this is not the case then, with $\chi $ denoting the spherical metric on the extended complex plane,
 there exists a small positive $\sigma$ such that
$f(z) \neq 0, \infty $ on $X = \{ z \in \C : \chi (z, z_0)  = \sigma \} $
and $z(t)$ meets $X$  infinitely often as $t \to b_0-$. 
But this gives $z_0' \in X$ 
such that $z_0'$ is a limit point of $z(t)$ as $t \to b_0-$, and hence a contradiction.

It remains only to note that if $z_0 $ is a zero of $f$ then it takes infinite time for any trajectory of (\ref{flow1}) to tend to $z_0$. 
To see this, take $C > 0$ and $m \in \N$ such that 
$|f(z)| \leq C |z-z_0|^m $ as $z \to z_0$. Let $n$ be large and take any trajectory $z(t)$ 
such that $|z(t_n) - z_0| = 2^{-n}$ and $|z(t_{n+1}) - z_0| = 2^{-n-1}$ 
and $2^{-n-1} \leq |z(t) - z_0| \leq 2^{-n}$ for $t_n \leq t \leq t_{n+1}$. This yields 
$$
2^{-n-1} \leq | z(t_{n+1})-z(t_n)| = \left| \int_{t_n}^{t_{n+1}} f(z(t)) \, dt \right| \leq (t_{n+1}-t_n) C 2^{-nm} 
$$
and so $t_{n+1}- t_n \geq C^{-1} 2^{(m-1)n-1} \geq 1/2C.$ 
\end{proof}

The remainder of this section will be occupied with the proof of the following.

\begin{prop}
 \label{prop1}
Let the function $f$ be transcendental and meromorphic in the plane, and assume the existence of an unbounded set 
$F_1 \subseteq [1, \infty)$ 
and a function $N(r): F_1 \to [1, \infty)$ 
with 
\begin{equation}
 \label{Nrdef}
\lim_{r \to \infty, r \in F_1} N(r) = \infty, 
\end{equation}
such that for each $r \in F_1$ there exists $z_r $ with $|z_r| = r$ and $f(z_r) \neq 0$ and 
\begin{equation}
 \label{Nlim4}
f(z) = (1+o(1))   \left( \frac{z}{z_r} \right)^{N(r)} f(z_r)   \quad \hbox{on} \quad  D(z_r, 8), 
\end{equation}
as $r \to \infty$ in $F_1$, where 
\begin{equation}
\label{Nlim1}
D(z_r, L) = 
\left\{  z_r e^\tau : \, 
\max \{ | {\rm Re} \, \tau  | , \, | {\rm Im } \, \tau | \} \leq L N(r)^{-5/8} \,\right\}.
\end{equation}
Then for all sufficiently large $r \in F_1$ there exist 
$Q \geq N(r)^{1/4} $  points $Y_1, \ldots, Y_Q $ in $ D(z_r, 1)$, each with the property that 
the trajectory $\gamma_j = \zeta_{Y_j}$ with $\zeta_{Y_j}(0) = Y_j$ of (\ref{flow1}) has maximal interval of definition 
$(\alpha_{Y_j}, \beta_{Y_j})$, where 
\begin{equation} 
\label{betajdef}
\beta_{Y_j} \leq P_r = \frac{ 2r  }{|f(z_r)| (N(r)-1) \exp( N(r) ^{1/4} ) } .
\end{equation}
These trajectories $\gamma_j$ are pairwise disjoint. 
\end{prop}

To prove Proposition \ref{prop1}, let  $r \in F_1$ be large, let $N = N(r)$ and  define $w_r$ by 
$w_r = z_r \exp \left(  4 N^{-5/8}  \right)$.
Then  (\ref{Nrdef}), (\ref{Nlim4}) and Cauchy's estimate for derivatives yield 
\begin{eqnarray}
 A(z) &=&  \frac1{f(z)}  = \left( \frac{z}{w_r} \right)^{-N} A(w_r) ( 1 + \mu (z) ) , \nonumber \\
\quad \mu (z) &=& o(1), \quad \mu'(z) = o \left( \frac{N^{5/8}}{r}  \right),
\label{Nlim3}
\end{eqnarray}
uniformly for $z$ in $D(z_r, 4)$. Again for $z$ in $D(z_r, 4)$, set 
\begin{eqnarray}
 Z &=& F(z)  = \frac{ w_r A(w_r)  }{1-N} + \int_{w_r}^z A(t) \, dt \nonumber \\
&=&
\frac{ w_r A(w_r)  }{1-N} + \int_{w_r}^z  \left( \frac{t}{w_r} \right)^{-N} A(w_r) ( 1 + \mu (t) ) \, dt ,
\label{Nlim5}
\end{eqnarray}
and let $\sigma_z$ be the path 
from $w_r$ to $z$ which consists of the radial segment from $w_r$ to
$\widehat z = w_r |z/w_r|$ followed by the shorter circular arc from $\widehat z$ to $z$. Then  $\sigma_z$ has 
length $O(r N^{-5/8} )$ and  $|w_r| \geq |t| \geq |z|$ on $\sigma_z$, so (\ref{Nlim3})  and integration by parts along $\sigma_z$ yield
\begin{eqnarray*}
\int_{w_r}^z t^{-N} \mu (t) \, dt &=&
o\left(  \frac{|z|^{1-N} }{N-1} \right) - \int_{w_r}^z o \left( \frac{N^{5/8}}{r}  \right) \, \frac{  t^{1-N} }{1-N} \, 
  dt = o\left(  \frac{|z|^{1-N} }{N-1} \right)  .
\end{eqnarray*}
Hence $Z$ satisfies, still for $z \in D(z_r, 4)$, using (\ref{Nlim4}) and (\ref{Nlim5}), 
\begin{eqnarray}
\label{b7}
Z &=& F(z) \sim \frac{  z^{1-N}A(w_r) }{ w_r^{-N} (1-N) }  
\sim \frac{  z^{1-N}A(z_r) }{ z_r^{-N} (1-N) } , \nonumber \\
\quad |Z| &\sim& \left| \frac{z}{r} \right|^{1-N} T_r, \quad T_r  = \frac{r |A(z_r)|}{N-1},
\end{eqnarray}
and 
\begin{equation}
 \label{b7a}
\log Z = (1-N) \log \frac{z}{z_r} + \log \frac{z_rA(z_r)}{1-N} + o(1), 
\end{equation}
where  $\log (z/z_r)$ is chosen so as to vanish at $z_r$, and 
$\log (z_rA(z_r)/(1-N))$ is the principal value.

\begin{lem}
 \label{lemwv}
Any sub-trajectory $\Lambda \subseteq D(z_r, 4)$ of the flow (\ref{flow1}) 
is a level curve on which ${\rm Im} \, F(z)$ is constant and
${\rm Re} \, F(z)$ increases in increasing time. If $\Lambda$ joins $w_0$ to $w_1$  then the time taken for the flow (\ref{flow1}) to traverse 
$\Lambda$ is 
$$
\int_{w_0}^{w_1} \frac{dt}{dz} \, dz = \int_{w_0}^{w_1} \frac1{f(z)} \, dz = F(w_1) - F(w_0).
$$
Let $Q = Q_r$ be the largest positive integer not exceeding $2 N^{1/4}$. Then 
provided $r \in F_1$ is  large enough there exists a domain $\Omega_r $, the closure of which lies in $D(z_r, 1)$, such that 
$Y = \log Z$ maps $\Omega_r$ univalently onto the rectangle
\begin{eqnarray}
 \label{Srdef}
G_r &=& 
\{ Y \in \C : \log S_r < {\rm Re} \, Y < \log T_r, \quad 0 < {\rm Im} \, Y < 4Q \pi \} , \nonumber \\
S_r &=& T_r \exp( - N^{1/4} ) = \frac{P_r}2 ,
\end{eqnarray}
and $S_r = o(T_r)$ as $r \to \infty$ with $r \in F_1$.  The boundary of $\Omega_r$ contains a
simple arc $L_r$ such that, as $z$ describes the arc $L_r$ once, the image $w=Z=F(z)$ describes $2Q$ times the circle 
$|w| = S_r$, starting from $w=S_r$.  
Moreover, $\Omega_r$ contains $2Q$  pairwise disjoint simply connected domains $V_r^1, \ldots , V_r^{2Q}$, 
each  mapped univalently by $F$ onto 
$\{ w \in \C : S_r  < |w| < T_r  , \, 0 < \arg w < 2 \pi \}.$
These domains  have the following additional properties.

Let $V_r$ be any one of the $V_r^j$. 
Then $\partial V_r$ consists of the following: two simple arcs  $I_r\subseteq L_r $  and $J_r $ 
mapped by $F$ onto the circles $|w| = S_r$ and $|w| = T_r$ respectively;
two  sub-trajectories of (\ref{flow1}) mapped by $F$ onto the interval $[ S_r, T_r]$.
\end{lem}
\begin{proof} 
The first two assertions hold because writing $Z = F(z)$ gives $\dot Z = 1$. 
The existence of $\Omega_r$, $L_r$ and the $V_r^j$ follows from (\ref{b7}) and (\ref{b7a}),  which imply 
that  $\log Z$ is a univalent function of $\log z$ on $D(z_r, 7/2)$. In particular, $L_r$ is the pre-image 
under $\log Z$ of  $\{ \log S_r + i \sigma : \, 0 \leq \sigma \leq 4 Q \pi \}$.  
Finally, (\ref{Nrdef}), (\ref{betajdef}), (\ref{Nlim3}), (\ref{b7}) and (\ref{Srdef}) give $P_r = 2S_r = o(T_r)$. 
\end{proof}

Assume henceforth that $r \in F_1$ is so large that Lemma \ref{lemwv} gives $P_r = 2S_r < T_r - S_r$. 
Choose some $V_r = V_r^j$ and let $W_r$ be the closure of $V_r$. 
The next lemma describes the behaviour of the trajectory $\zeta_w(t)$ of (\ref{flow1}) through $\zeta_w(0) = w \in  I_r$. 
\begin{lem}
 \label{lemexit}
Suppose that $w \in I_r$ 
and ${\rm Re} \, F(w) \geq 0$. Then there exists $t_w \geq T_r - S_r$ 
such that $\zeta_w(t) \in W_r \setminus (J_r \cup I_r)$ for $0 < t < t_w$, while 
$\zeta_w(t_w) \in J_r$.
If ${\rm Re} \, F(w) >  0$ and $t < 0$ and $|t| $ is small, then $|F(\zeta_w(t))| < S_r $. 

Similarly, if $w \in I_r$ 
and ${\rm Re} \, F(w) \leq 0$, 
there exists $t_w \leq S_r - T_r$ such that $\zeta_w(t) \in W_r \setminus (J_r \cup I_r)$ for $t_w < t < 0$, while 
$\zeta_w(t_w) \in J_r$.
If ${\rm Re} \, F(w) <  0$ and $t > 0$  is small, then $|F(\zeta_w(t))| < S_r $.
If $w \in I_r$ 
and ${\rm Re} \, F(w) = 0$, then $\zeta_w(t)$ travels from $w$ to $J_r$ via $W_r$ in both increasing and decreasing time.

\end{lem}
\begin{proof} Let $w \in I_r$ 
and ${\rm Re} \, F(w) \geq 0$. Then $|F(w)| = S_r$ and, for small positive $t$, both of 
${\rm Re} \, F(\zeta_w(t))$ and $|F(\zeta_w(t))|$  are increasing, while 
${\rm Im} \, F(\zeta_w(t))$ is constant; thus $\zeta_w(t)$ remains within $W_r$ until it exits via $J_r$. 
The time taken to pass from $w$ to the first encounter with $J_r$, at $W$ say, is 
$F(W) - F(w) = | F(W)-F(w) | \geq T_r - S_r$.
The remaining assertions are proved similarly.
\end{proof}

\begin{definition}
 \label{def0}
For $u \in \C$ 
let $u^*$ denote the reflection of $u$ across the imaginary axis. 
A point $w \in I_r$ will be called \textit{recurrent} if 
${\rm Re} \, F(w) < 0$ and there exists $t' > 0$ such that: (i) $\zeta_w(t)$ is defined for $0 \leq t \leq t'$
and  $w' = \zeta_w(t') \in I_r$; (ii) $F(w') = F(w)^*$; (iii) 
$\zeta_w(t) \not \in  L_r $ for $0 < t < t'$; (iv) the Jordan curve $\Gamma_w$, formed from  the arc of 
$I_r$ joining $w$ to $w'$ and the sub-trajectory $\zeta_w(t)$, $0 \leq t \leq t'$, encloses no zeros and no poles of $f$. 

\end{definition}
Since $F$ is univalent on $V_r$, and maps $I_r$ onto the circle $|w| = S_r$, with the end-points of $I_r$ mapped to $S_r$, it follows that
for $w, w' \in I_r$ the equation $F(w') = F(w)^*$ determines $w'$ uniquely from $w$, except when $F(w) = - S_r$.
The next lemma follows at once from Lemma \ref{lemwv} and Cauchy's theorem applied to $1/f$ and $\Gamma_w$. 
\begin{lem}
 \label{lemnever}
If $w \in I_r$ is recurrent then $t' \leq | F(w') - F(w) | \leq 2 S_r = P_r$. 
\end{lem}

\begin{lem}
 \label{lemrecur}
If $w \in I_r$ with ${\rm Re} \, F(w) < 0$ and $F(w)$ close to $ \pm i S_r$, then
$w$ is recurrent. 
\end{lem}
\begin{proof}
By the construction of $V_r$, the point $w$ lies in a small neighbourhood $\widehat U$ of some $\widehat w \in I_r$ with $F\left(\widehat w \right) = \pm i S_r$ and 
$F$ univalent on $\widehat U$.
Hence, 
as $\zeta$ describes $\zeta_w$ in increasing time, the image $F(\zeta)$ traverses the horizontal chord from  $F(w)$ to $F(w)^*$
and $\zeta$ remains within $\widehat U$; thus
$\zeta $ returns to meet $ L_r$ at
$w' \in I_r$ with $F(w') = F(w)^*$. 
Therefore $w $ is recurrent.
\end{proof}

Lemma \ref{lemrecur} implies that the set of recurrent $w \in I_r$ is non-empty, and it follows from the next lemma 
that, for all but at most two $V_r^j$, the absence of $Y_j$ as in the conclusion of Proposition \ref{prop1}
forces all $v \in I_r$ with ${\rm Re} \, F(v) < 0$ to be recurrent.

\begin{lem}
 \label{Jrclosed}
Let $V_r = V_r^j$  be such that neither end-point of the arc $L_r$ lies in $W_r$, and assume that 
no $y \in I_r$  is such that ${\rm Re} \, F(y) < 0$ and $\zeta_y$ 
has maximal interval of definition $(\alpha_y, \beta_y)$ with $\beta_y \leq P_r = 2S_r$. Then the following statements hold.\\
(a) Let $w \in I_r$ be such that ${\rm Re} \, F(w) < 0$  and there exists a sequence $(w_n)$ in $I_r$ for which
$w_n \to w$ as $n \to \infty$ and each $w_n$ is recurrent.
Then $w$ is recurrent and $w_n' \to w'$ as $n \to \infty$. \\
(b) All $v \in I_r$ with ${\rm Re} \, F(v) < 0$ 
are recurrent. 
\end{lem}
\begin{proof} Let $w$ be as in (a), and observe that $F(w)^* \neq F(w)$, since  ${\rm Re} \, F(w) < 0$, and that
$\zeta_w(t) \not \in L_r $ for small positive~$t$, by  
Lemma \ref{lemexit}. 
By assumption, $\zeta_w$ 
has maximal interval of definition $(\alpha_w, \beta_w)$ with $\beta_w > 2S_r$.

Suppose first that there exists $\delta $ such that  
\begin{equation}
 \label{notIr}
|\zeta_w(t) - u| \geq 2 \delta > 0 \quad \hbox{for all $u \in I_r$ with $F(u) = F(w)^*$ and  all $t \in [0 ,  2 S_r]$. } 
\end{equation}
Note here that there exist at most two $u \in I_r$ with $F(u) = F(w)^*$. 
Since $w_n \to w$ and $w_n$ is recurrent
it follows that $F(w_n') = F(w_n)^* \to F(w)^*$, and so $w_n'$, for each large $n$, is close to some 
$u \in I_r$ with $F(u) = F(w)^*$.  But  (\ref{notIr}) and
continuous dependence on starting conditions now imply that if $n$ is large then 
\begin{equation*}
 \label{notnear}
| \zeta_{w_n}(t) - w_n'| \geq \delta \quad \hbox{for $0 \leq t \leq 2S_r$. }
\end{equation*}
This contradicts the fact that 
Definition \ref{def0} and
Lemma~\ref{lemnever} give $w_n' = \zeta_{w_n} (t_n') $, where $ 0 < t_n' \leq 2S_r$.  
Hence (\ref{notIr}) cannot hold, and  there exists a minimal  $s $ with 
\begin{equation}
 \label{that1}
0 < s \leq 2S_r, \quad W= \zeta_w(s) \in L_r  ,
\end{equation}
because if this is not the case then (\ref{notIr}) evidently holds for some choice of  $\delta $. 

Suppose that ${\rm Re} \, F(W) \leq 0$, and take $k$ (possibly with  $k \neq j$) such that $W \in \partial V_r^k$.
Since $s \leq 2S_r < T_r - S_r$, applying 
Lemma~\ref{lemexit} to this $V_r^k$ shows that $w = \zeta_w(0) = \zeta_{W} \left( \, - s \, \right) \not \in L_r$, a contradiction.

Thus $W = \zeta_w( s)  \in L_r$ and ${\rm Re} \, F(W)$ is positive.
Suppose that $W \not \in I_r$ or $F(W) \neq F(w)^*$, and take any $u \in I_r$ with $F(u) = F(w)^*$. Then Lemma \ref{lemexit} (applied possibly to a different
$V_r^k$) and the minimality of $s$ in (\ref{that1}) give
$\zeta_w(t) \neq u$ 
for $0 \leq t \leq x = s + T_r - S_r $. Since $x > 2S_r$ there must exist $\delta $ such that (\ref{notIr}) holds, which is impossible.
This proves that $W = \zeta_w( s)  \in I_r$ and $F(W) = F(w)^*$, so that $w$ satisfies conditions (i) to (iii) of Definition \ref{def0},
with $t' = s$ and $w' = W$.

Now take any sequence $(x_n) $ in $I_r$ with $x_n \to w$ as $n \to \infty$. 
The trajectory $\zeta_w $ meets $L_r$ non-tangentially at $w$ and $W$, because $|F(z)| = S_r$ on $L_r$ and $Z = F(z)$ gives
$\dot Z = 1$ locally. 
Take a small positive $\rho $ 
and let $n \in \N$ be large. Then $\zeta_w(t)$ does not meet $L_r$ for $\rho \leq t \leq s - \rho $, 
by the minimality of $s$  in (\ref{that1}), 
and nor does $\zeta_{x_n}(t)$, by continuous dependence on 
initial conditions. Moreover, for $0 \leq t \leq \rho$, the trajectory
$\zeta_{x_n}(t)$ follows a level curve on which ${\rm Im} \, F$ is constant, from $x_n$ to 
$\zeta_{x_n}(\rho)$, in which $F(\zeta_{x_n}(\rho)) = F(x_n) + \rho $. Furthermore,
$\zeta_{x_n} (s - \rho )$ is close to $\zeta_w( s - \rho ) $, which satisfies
$F(\zeta_w( s - \rho )) = F( W) - \rho $. Thus for $t - s + \rho $ small and positive, $\zeta_{x_n}(t)$ again follows a level curve
of ${\rm Im} \, F$, meeting $L_r$ non-tangentially at some point $x_n''$ near to $W$, using the fact that $W$ is not an end-point of $L_r$. 
Therefore $\zeta_{x_n} $
follows close to $\zeta_w$ and returns for the first time to $L_r$ at $x_n''$.

Applying this argument with $x_n = w_n$ shows that  $w_n' = x_n'' \to W = w'$, and that if $\Gamma_w$ is as in Definition \ref{def0} then, for large $n$,
each point of $\Gamma_{w_n}$ lies close to  $\Gamma_w$. Thus 
$w$ also satisfies condition (iv), and is recurrent. This proves part  (a).

To prove part (b), observe that  $I_r$ has relatively open subsets $U^+$, $U^-$, mapped by $\arg F(z)$ onto 
$(\pi/2, \pi )$ and $(-\pi, - \pi/2)$ respectively. Let $U_0$ be one of $U^+$, $ U^-$; then 
$\widehat U_0 = \{ w \in U_0: \hbox{$w$ is recurrent} \} \neq \emptyset $, by Lemma \ref{lemrecur}. 
Suppose that $U_0 \neq \widehat U_0$. Then 
there exists some $v \in U_0$ which is a boundary point of $\widehat U_0$ relative to $U_0$; thus 
$v \in I_r$ with ${\rm Re} \, F(v) < 0$ and $F(v) \not = - S_r$ 
and there are sequences $w_n \to v$, $v_n \to v$, with $w_n, v_n \in I_r$, such that
each $w_n$ is recurrent, while each $v_n$ is not. 
By (a), $v$ is recurrent. For large $n$ the   argument in the proof of (a), with $x_n = v_n$,  $w=v$ and $W = v'$,
shows that $\zeta_{v_n}$  returns to meet $ L_r$ for the first time after leaving $v_n$, 
at some  $u_n = x_n'' \in I_r$ close to $v'$, without looping around any zeros or poles of $f$. 
But then Cauchy's theorem gives 
${\rm Im} \, (F(u_n) - F(v_n)) = 0$ and 
$F(u_n) = F(v_n)^*$, so that $v_n$ is recurrent, a contradiction. Hence all $v \in I_r$ with ${\rm Re} \, F(v) < 0$ and $F(v) \neq -S_r$
are recurrent, and the same holds when $F(v) = -S_r$, by  part (a). 
\end{proof}

\begin{lem}
 \label{Jrclosed3}
Let $V_r = V_r^j$  be such that neither end-point of the arc $L_r$ lies in $W_r$. Then there exists  $y \in I_r$ such that ${\rm Re} \, F(y) < 0$ and $\zeta_y$ 
has maximal interval of definition $(\alpha_y, \beta_y)$ with $\beta_y \leq P_r = 2S_r$.
\end{lem}
\begin{proof} Assume that this is not the case, and consider the unique $w \in I_r$ with $F(w) = -S_r$. Then Lemma \ref{Jrclosed}
shows that $w$ is recurrent, and so $w'$ is one of the two
points $u_1, u_2$ on $I_r$ with $F(u_j) = S_r$; label these so that $w' = u_1$. 
Choose a sequence $v_n \in I_r$ with $v_n \to u_2$, $ v_n \neq  u_2$, and for large $n$ choose the unique $w_n \in I_r$ with $F(w_n) = F(v_n)^* \to F(u_2)^* = -S_r$.
Thus $w_n \to w, w_n \neq w$, and $w_n$ is 
recurrent for large $n$, by Lemma \ref{Jrclosed}. But this gives 
$w_n' = v_n \to u_2 \neq  w'$,  contradicting Lemma \ref{Jrclosed}.
\end{proof}

It follows from Lemma \ref{Jrclosed3} that, for large $r \in F_1$, 
at least $2Q - 2 \geq Q \geq N^{1/4}$ of the domains  $ V_r^1, \ldots , V_r^{2Q}$  give rise to pairwise distinct $ Y_j \in \partial V_r^j \cap L_r$ such that  
${\rm Re} \, F(Y_j) < 0$ and $\zeta_{Y_j}$ 
has maximal interval of definition $(\alpha_{Y_j}, \beta_{Y_j})$, in which $\beta_{Y_j} $ satisfies (\ref{betajdef}). 
Suppose that 
these trajectories are not pairwise disjoint. Then there exist distinct $j$ and $k$ such that $Y_j = \zeta_{Y_k}(S) $ and $Y_k = \zeta_{Y_j}(-S) \in L_r$ for some $S$ 
with $0 < S < P_r$. But Lemma~\ref{lemexit} shows that 
$\zeta_{Y_j}(t) \not \in L_r$ for $S_r - T_r < t < 0$, and $T_r - S_r > 2 S_r = P_r$, a contradiction. 
Proposition \ref{prop1} is proved.

\section{Proof of Theorem \ref{thm0}}\label{nonzero}

Let $f$ be a transcendental meromorphic function in the plane with finitely many  poles. 
Write  $ f = B/C$, where
$B$ is a transcendental entire function and $C $ is a polynomial, having no zeros in common with $B$.
The Wiman-Valiron theory \cite{Hay5} may now be applied to $B$ as follows. Starting from the Maclaurin series  $B(z) = \sum_{k=0}^\infty b_k z^k $ of $B$,
the central index  $N(r) = \nu(r, B)$ is defined for $r \geq 0$ to be the largest integer $n $ such that
$|b_n| r^n = \max_{k} |b_k| r^k $, and $N(r)$ tends to infinity with $r$.
For large $r > 0 $ choose $z_r$ with 
$|z_r| = r$ and 
$ |B(z_r)| = M(r, B) = \max \{ |B(z)| : |z| = r \} $. 
Then  \cite[Theorem 10]{Hay5} gives  $F_1 \subseteq [1, \infty)$ 
such that $[1, \infty) \setminus F_1$ has finite logarithmic measure and 
\begin{equation*}
\frac{f(z)}{f(z_r)} \sim \frac{B(z)}{B(z_r)}
\sim   \left( \frac{z}{z_r} \right)^{N(r)}   \quad \hbox{on} \quad  D(z_r, 8), 
\end{equation*}
as $r \to \infty$ in $F_1$, where $D(z_r, 8)$ is given by (\ref{Nlim1}).

Now Lemma \ref{limptlemmero} and 
Proposition \ref{prop1} give an arbitrarily large number of pairwise disjoint trajectories for (\ref{flow1}), 
each tending to infinity or a pole of $f$ in finite
increasing time. 
But each of the finitely many poles of $f$ has only finitely many
trajectories tending to it in increasing time (see Section \ref{prelim}).
This proves Theorem \ref{thm0}.

It seems conceivable that  the conclusion of Theorem \ref{thm0} would remain true for all meromorphic functions $f$ in the plane
such that the inverse function $f^{-1}$ has a direct transcendental
singularity over $\infty$  \cite{BE}. This is a weaker hypothesis than those of Theorems \ref{thm0} and \ref{thm2}, and means that
there exist $M > 0$ and a component $U$ of the set 
$\{ z \in \C : |f(z)| > M \}$ which contains no poles of $f$, but does contain a path tending to infinity on which $f(z)$ tends to infinity.
In this case, Theorems 2.1 and 2.2 of \cite{BRS} give $F_1$ and $N(r)$ such that (\ref{Nrdef}) and (\ref{Nlim4}) are satisfied, where $|z_r| = r$,
$D(z_r, 8) \subseteq U$ and 
$\log r = o( \log^+ |f(z_r)| )$ as $r \to \infty$ in $ F_1$, 
while $[1, \infty) \setminus F_1$ has finite logarithmic measure. Thus Proposition \ref{prop1} may be applied, with $P_r  \to 0$ 
as $r \to \infty$ in $ F_1$, by (\ref{betajdef}), 
but in general it seems difficult to exclude the 
possibility that all the trajectories $\zeta_{Y_j}$ thereby obtained  tend to  poles of $f$. It is true, however, that if such a trajectory does tend to a pole then 
it must exit $U$ and subsequently enter another component $U_r$ of $\{ z \in \C : |f(z)| > M \}$, giving rise to an interval 
$[t_1, t_2] \subseteq (0, \beta_{Y_j}) \subseteq (0, P_r)$ on which $|f(\zeta_{Y_j} (t))| \leq M$, with
$\zeta_{Y_j} (t_1) \in \partial U$ and $\zeta_{Y_j} (t_2) \in \partial U_r$.
Hence the distance from $U$ to $U_r$ is at most $M(t_2 - t_1) \leq 
M P_r $, which for large $r \in F_1$ is extremely small. 
Such a component $U_r$ cannot exist if, for example, $f(z) = g(z) \tan z$, where $g$ is a transcendental entire function which is bounded on the  strip
$\{ z \in \C : \, | {\rm Im} \, z | \leq T \}$, for some $T > 0$; in this case $f^{-1}$ has a direct transcendental
singularity over $\infty$  and (\ref{flow1}) has infinitely many 
trajectories tending to infinity in finite increasing time.

\section{Proof of Theorem \ref{thm1}}
Let $f $, $\Omega$, $M$, $U$, $F$ and $H$ be as in the hypotheses. 
It may be assumed that $M=1$: if this is not the case then (\ref{flow1}) and  $\Omega$ may be re-scaled by writing
$w = z/M$ and $\dot w = f(z)/M = g(w)$. 
Let $z = \phi (v)$ be the inverse function of $F$, mapping
$H = \{ v \in \C : {\rm Re}\, v > 0 \}$ univalently onto  $ U$, and on $H$ consider the flow 
\begin{equation}
 \label{newflow}
\phi'(v) \dot v = e^v  .
\end{equation}
The essence of the proof lies in showing that, since $\phi'(v)$ varies relatively slowly on $H$, there are trajectories 
for (\ref{newflow}) in $H$
which tend to infinity in finite time, and these are mapped via $z = \phi (v)$ to trajectories of (\ref{flow1}) which tend to the
extended boundary of $\Omega$. 

For $v \in H$ the function 
$$
h(u) = \frac{\phi( v + u {\rm Re}\, v) - \phi(v)}{\phi'(v) {\rm Re}\, v } = u + \sum_{n=2}^\infty a_n u^n 
$$
is univalent for $|u| < 1$, so that Bieberbach's theorem gives  $|h''(0)| = 2 |a_2| \leq 4$ and
\begin{equation}
 \label{4a}
\left| \frac{\phi''(v)}{\phi'(v)} \right| \leq \frac4{{\rm Re}\, v} \quad \hbox{and} \quad 
\left| \log \left( \frac{\phi'(s)}{\phi'(v)} \right)  \right| \leq \frac{C_0 R}{ {\rm Re}\, v }
\quad \hbox{for} 
\quad |s - v| < R < \frac{{\rm Re}\, v }2 ,
\end{equation}
where $C_0$ is a positive absolute constant. Moreover, there exists $C_1 > 0$ with 
\begin{equation}
 \int_{[v, + \infty )} e^{-t} | \phi'(t) | \, dt 
\leq 
\frac{|\phi'(v)|}{v^4} \int_{[v, + \infty )} t^4  e^{-t}  \, dt 
\leq C_1 e^{-v} |\phi'(v)|
 \label{7}
\end{equation}
for $v \in [1, + \infty ) \subseteq \R$.
Therefore, for $w \in H$, Cauchy's theorem and (\ref{7}) lead to
\begin{eqnarray}
 \label{10a}
D &=& \int_{[1, + \infty )} e^{-t} \phi'(t) \, dt \in \C, \nonumber \\
\int_1^w e^{-t} \phi'(t) \, dt  &=& D - \psi (w) =
D - \int_w^{+ \infty}  e^{-t} \phi'(t) \, dt .
\end{eqnarray}
Here  the  integral from $1$ to $w$ is along any piecewise smooth contour in $H$, while that from $w$ to $+\infty$ is eventually along 
an  interval $[M_w, + \infty)$ with $M_w \geq 1$, and $\psi(w)$ is analytic on $H$.

Let   $N_1$ and $N_2/N_1$ be large and positive, and  for $j=1, 2$ 
let $H_j$ denote the convex domain 
\begin{equation*}
 \label{omega1def}
H_j = \left\{ x + iy: x > N_j, \, -  x^{1/2j} < y <  x^{1/2j}  \right\} \subseteq H.
\end{equation*}
Let $w$ lie in  $H_1 $, and write
\begin{equation}
 \label{11a}
x = {\rm Re} \, w, \quad s = x + \sqrt{x} . 
\end{equation}
Then (\ref{4a}), (\ref{7}) and (\ref{11a}) imply that
\begin{eqnarray}
 \label{12a}
\phi'(w) &\sim& \phi'(x) \sim \phi'(s) , \nonumber \\ 
\left| \int_{[s, + \infty )} e^{-t} \phi'(t) \, dt \right| &\leq& C_1 e^{-s} | \phi'(s)| 
= o ( | e^{-w} \phi'(x) |). 
\end{eqnarray}
Further, the integral over the  line segment from $w$ to $s$ satisfies, by (\ref{4a}) and (\ref{11a}), 
\begin{equation}
 \label{13a}
 \int_w^s e^{-t} \phi'(t) \, dt  = \phi'(x) \int_w^s e^{-t} (1 + o(1)) \, dt 
= \phi'(x) ( e^{-w} - e^{-s} + \eta (w))  ,
\end{equation}
in which parametrising with respect to $\rho = {\rm Re} \, t$ gives
$$
 |\eta (w) |  =  \left| \int_w^s e^{-t} o(1) \, dt  \right|  \leq o(1) \int_x^s e^{-\rho} \, d \rho = o(e^{-x} ) = o( |e^{-w}|) .
$$
Combining the last estimate with (\ref{4a}), (\ref{10a}), (\ref{11a}), 
(\ref{12a}) and (\ref{13a})  leads to
\begin{equation}
 \label{15a}
\psi(w) \sim e^{-w} \phi'(x) , \quad 
\lambda (w) = -\log \psi(w) 
= w + O( \log |w| ) 
\end{equation}
as $w \to \infty$ in $H_1$. 
Since $N_2/N_1$ is large, (\ref{4a}), (\ref{15a}) and Cauchy's estimate for derivatives yield $| \lambda'(w) - 1 | < 1/2$ on $H_2$, which implies that 
$\lambda (w)$ is univalent on $H_2$.
Let $N_3$ and $N_4$ be positive integers with $N_3/N_2$ and $N_4/N_3$ large. 
Then (\ref{15a}) shows that
for $j=0, \ldots, N_3$ there exists a simple path $L_j$ tending to infinity in $H_2$ and mapped by $\lambda$ onto the path
$
\{ j 2 \pi i + t: t \geq N_4 \}.
$
Thus  $\psi = e^{- \lambda } $ maps each $L_j$ injectively onto $(0, h]$, where $h = e^{-N_4} > 0$; moreover
$\psi(v) \to 0$ and $e^v \to \infty$ as $v \to \infty$ on $L_j$.

Parametrise  $L_j \subseteq H_2 \subseteq H$ by $w = v(s)$, where 
$- \psi(v(s)) = s$ for $-h \leq s < 0$.
Thus
$$
1 = - \psi'(v(s)) \frac{dv}{ds} = e^{-v(s) } \phi'(v(s)) \frac{dv}{ds}, \quad 
$$
using (\ref{10a}),
and so there exist  $N_3$ pairwise disjoint
trajectories $L_j$ in $H$ of the flow (\ref{newflow}), on which $v$ and $e^v$ tend to infinity as $s \to 0-$ and so
in finite increasing time.

Thus  the flow (\ref{newflow}) has infinitely many disjoint trajectories $L$ in $H$, on each of
which  $v(t)$ and $e^{v(t)}$ tend to infinity in finite increasing time. Because
$\phi$ is univalent, these  trajectories have disjoint images under $\phi$ in $U$. For each such trajectory $L$, write
$$
z = \phi (v), \quad 
\dot z = \phi'(v) \dot v = e^v = e^{F(z)} = f(z). 
$$
Thus $f(z(t))$ tends to infinity in finite increasing time along $\phi(L) \subseteq U$, and it remains only to show that $z(t)$ tends to the
extended boundary of $\Omega$. 
Assume that this is not the case: then there exists a sequence $(v_j) \subseteq L$ such that $e^{v_j}$ tends to infinity
but $\beta_j = \phi( v_j) \to \beta_0 \in \Omega$ as $j \to \infty$. Because $f( \beta_j ) = e^{v_j} \to \infty$, it must be the case
that $\beta_0 $ is a pole of $f$ in $\Omega$. But then there exist a large positive $M_1$ and a neighbourhood $U_1 $ of
$\beta_0$ such that the closure of
$U_1$ lies in $\Omega$ and
$f$ maps $U_1 \setminus \{ \beta_0 \}$ finite-valently onto $\{ w \in \C : M_1 < |w| < + \infty \}$. 
For large $j$ the line ${\rm Re} \, v =  {\rm Re} \, v_j $ is mapped by $z = \phi (v)$ onto a level curve $\Gamma \subseteq U $ 
on which $|f|$ is constant, and $\Gamma$ passes through $\beta_j \in U_1$ and so must lie wholly in $U_1$. 
On the other hand, by the univalence of~$\phi$, the level curve $\Gamma$ contains infinitely many  distinct points 
$\phi ( v_j + k 2 \pi i )$, $k \in \Z$, each satisfying 
$f( \phi (v_j + k 2 \pi i )) = e^{v_j} = f(\beta_j)$. 
This proves Theorem \ref{thm1}. 

\textit{Acknowledgement.}
 The author thanks the referee for 
a very careful reading of the manuscript and 
extremely helpful suggestions.

\footnotesize{
}

\end{document}